\documentclass[12pt]{amsart}
\usepackage{amssymb}
\usepackage{enumerate}

\makeatletter
\@namedef{subjclassname@2010}{%
  \textup{2010} Mathematics Subject Classification}
\makeatother

\newtheorem{theorem}{Theorem}[section]
\newtheorem{lemma}[theorem]{Lemma}

\numberwithin{equation}{section}

\newtheorem{conjecture}[theorem]{Conjecture}

\newcommand{\Z}{\mathbb{Z}}

\begin{document}


\baselineskip=17pt



\title{On compositeness of special types of integers}

\author{Yu Tsumura}
\address{Department of Mathematics, Purdue University
150 North University Street, West Lafayette, Indiana 47907-2067
}
\email{ytsumura@math.purdue.edu}

\date{}

\begin{abstract}
In paper on a classification of Lehmer triples, Juricevic conjectured that there are infinitely many primes of special form.
We disprove one of his conjectures and consider the other one.
\end{abstract}



\maketitle

\section{Introduction}
In paper \cite{Juricevic}, Juricevic proved a theorem on a classification of Lehmer triples under the assumption of the following two conjectures.
\begin{conjecture}\label{conj1}
There are infinitely many prime numbers $p>5$ such that
\[T(p):=\frac{1}{5}\left( (1+\sqrt{5}) \left( \frac{3+\sqrt{5}}{2} \right)^{2p}+(1-\sqrt{5}) \left( \frac{3-\sqrt{5}}{2} \right)^{2p}+3\right)\]
is a prime number.
\end{conjecture}

\begin{conjecture}\label{conj2}
There are infinitely many prime numbers $p>5$ such that
\[Y(p):=\frac{1}{3}\left( (1+\sqrt{3}) \left( 2+\sqrt{3} \right)^{2p}+(1-\sqrt{3}) \left( 2-\sqrt{3} \right)^{2p}+1 \right)\]
is a prime number.
\end{conjecture}
In this article we discuss those two types of integers.
Actually, we disprove Conjecture \ref{conj2} and also show that there are infinitely many composite numbers $T(p)$.
The proofs are elementary.
We use simple properties of second-order linear recurrence sequences.

\section{Lemmas}
Here we prove some basic properties of second-order linear recurrence sequences.

First of all, let us set up the situation.
Let $P$, $Q\in\Z\setminus\{0\}$ and $X^2-PX+Q=(X-\alpha)(X-\beta)$.
Hence $\alpha+\beta=P$ and $\alpha\beta=Q$.
Let us define $V_n(P,Q)=\alpha^n+\beta^n$ for integer $n\geq 0$.
Note that we have $V_0(P,Q)=2$ and $V_1(P,Q)=P$.

For simplicity, let $V_n$ denote $V_n(P,Q)$ when $P$, $Q$ are understood.
It is easy to show the next lemma and we omit the proof.
The reader will find many properties of second-order linear recurrence sequences, for example, in \cite{Ribenboim}.
\begin{lemma}\label{lem1}
For $P$, $Q\in\Z\setminus\{0\}$ and $V_n=V_n(P,Q)$, we have
\[V_n=PV_{n-1}-QV_{n-2}.\]
\end{lemma}

Now take $t_0$, $t_1\in \Z\setminus\{0\}$.
Let $\gamma=t_1-t_0\beta$ and $\delta=t_1-t_0\alpha$.
\begin{lemma}\label{lem2}
We have 
\[\gamma \alpha^n+\delta \beta^n=t_1V_n-Qt_0V_{n-1}.\]
\end{lemma}
\begin{proof}
We have 
\begin{eqnarray*}
\gamma \alpha^n+\delta \beta^n &=& t_1\alpha^n-t_0\alpha^n\beta+t_1\beta^n-t_0\alpha\beta^n \\
&=& t_1(\alpha^n+\beta^n)-t_0\alpha\beta(\alpha^{n-1}+\beta^{n-1})\\
&=& t_1V_n-t_0QV_{n-1}.
\end{eqnarray*}

\end{proof}

\section{All $Y(p)$ are composite.}
Now we disprove Conjecture \ref{conj2}.
\begin{theorem}
$Y(p)$ is composite for all primes $p$.
\end{theorem}
\begin{proof}
We use the lemmas with $P=4$, $Q=1$, $t_0=1$, $t_1=3$.
For these choices, we have $\alpha=2+\sqrt{3}$, $\beta=2-\sqrt{3}$, $\gamma=1+\sqrt{3}$ and $\delta=1-\sqrt{3}$.
By Lemma \ref{lem2}, we have for  integer $n$
\[3V_n-V_{n-1}=\gamma \alpha^n+\delta\beta^n
=(1+\sqrt{3}) \left( 2+\sqrt{3} \right)^{n}+(1-\sqrt{3}) \left( 2-\sqrt{3} \right)^{n}.\]
Hence we have for prime $p$
\begin{equation}\label{eq:Y}
Y(p)=\frac{1}{3}(3V_{2p}-V_{2p-1}+1).
\end{equation}

Now we prove that $Y(p)$ is divisible by 3 when $p \equiv 5 \pmod 6$ and is divisible by $13$ when $p \equiv 1 \pmod {6}$.
Since $Y(2)=3\cdot59$ and $Y(3)=23\cdot107$, this shows that $Y(p)$ is composite for all primes $p$.

We calculate $V_n \pmod 9$ using Lemma \ref{lem1}, that is, $V_n=4V_{n-1}-V_{n-2}$. 
The result is the following list.
\begin{center}
    \begin{tabular}{ | l | l | l | l| l|}
    \hline
    $n  $      & $V_n \pmod 9$          \\ \hline
     $0$           &  $2$         \\ \hline
    $1$           &  $4$      \\ \hline
    $2$           &  $5$             \\ \hline
    $3$           &  $7$      \\ \hline
    $4$           &  $5$      \\ \hline
    $5$           &  $4$           \\ \hline   
    $6$           &  $2$           \\ \hline 
    $7$           &  $4$           \\ \hline 
    \end{tabular}
\end{center}
Now since $V_n$ is calculated from the previous two terms, we see from the list that $V_n \pmod 9$ has period $6$.

Let $p\equiv 5 \pmod 6$.
Then $2p\equiv 4 \pmod 6$ and $2p-1\equiv 3 \pmod 6$.
Hence  by (\ref{eq:Y}) we have
\[3Y(p)\equiv 3V_4-V_3+1 \equiv 3\cdot5-7+1\equiv 0 \pmod 9.\]
Hence $3$ divides $Y(p)$ when $p\equiv 5 \pmod 6$.

Next we calculate $V_n \pmod {13}$ and obtain the following list.
\begin{center}
    \begin{tabular}{ | l | l | l | l| l|}
    \hline
    $n  $      & $V_n \pmod {13}$  &  $n  $      & $V_n \pmod {13}$       \\ \hline
     $0$           &  $2$       & $7$    &  $9$                       \\ \hline
    $1$           &  $4$        & $8$   & $12$                                   \\ \hline
    $2$           &  $1$         & $9$ & $0$                                    \\ \hline
    $3$           &  $0$         & $10$ & $1$                                     \\ \hline
    $4$           &  $12$       & $11$ & $4$                                    \\ \hline
    $5$           &  $9$        & $12$ & $2$                                  \\ \hline   
    $6$           &  $11$       & $13$ & $4$                                    \\ \hline 
                                           
    \end{tabular}
\end{center}
We see that $V_n \pmod {13}$ has  period $12$.

Let $p \equiv 1 \pmod {6}$.
Then  we have $2p \equiv 2 \pmod {12}$ and $2p-1\equiv 1 \pmod {12}$.
Hence by (\ref{eq:Y}) we have
\[3Y(p)\equiv 3V_2-V_1+1 \equiv 3\cdot 1 -4+1\equiv0 \pmod {13}.\]
It follows that $Y(p)$ is divisible by $13$ when $p \equiv 1 \pmod {6}$.

As noted above, this shows that $Y(p)$ is composite for all primes $p$.
\end{proof}

\section{There are infinitely many composite $T(p)$.}
Let us consider $T(p)$.
Now $T(p)$ is not composite for all primes $p$.
For example, $T(p)$ is prime for $p=2$, $5$, $809$.
(These are the only  primes the author found.)

Although we neither prove nor disprove Conjecture \ref{conj1}, we can show that there are infinitely many primes $p$ such that $T(p)$ is composite.
\begin{theorem}\label{thm:T}
Let $p$ be a prime number.
\begin{enumerate}
	\item If $p \equiv 1\pmod {5}$, then $T(p)$ is divisible by $5$.
	\item If $p \equiv 3\pmod {5}$, then $T(p)$ is divisible by $11$.
	\item If $p \equiv 2\pmod {15}$, then $T(p)$ is divisible by $31$.
\end{enumerate}
\end{theorem}
\begin{proof}
We use the above lemmas with $P=3$, $Q=1$, $t_0=2$, $t_1=4$.
With these values, we have $\alpha=(3+\sqrt{5})/2$, $\beta=(3-\sqrt{5})/2$, $\gamma=1+\sqrt{5}$ and $\delta=1-\sqrt{5}$.
By Lemma \ref{lem2}, it follows that for  integer $n$
\[4V_{n}-2V_{n}=\gamma \alpha^n+\delta \beta^n=(1+\sqrt{5}) \left( \frac{3+\sqrt{5}}{2} \right)^{n}+(1-\sqrt{5}) \left( \frac{3-\sqrt{5}}{2} \right)^{n}.\]
Hence we have 
\begin{equation}\label{eq:T}
T(p)=\frac{1}{5}(4V_{2p}-2V_{2p-1}+3)
\end{equation}
We calculate $V_n \pmod {25}$ using Lemma \ref{lem1}, that is, $V_n=3V_{n-1}-V_{n-2}$.
\begin{center}
    \begin{tabular}{ | l | l | l | l| l|}
    \hline
    $n  $      & $V_n \pmod {25}$  &  $n  $      & $V_n \pmod {25}$       \\ \hline
     $0$           &  $2$       & $6$    &  $22$                       \\ \hline
    $1$           &  $3$        & $7$   & $18$                                   \\ \hline
    $2$           &  $7$         & $8$ & $7$                                    \\ \hline
    $3$           &  $18$         & $9$ & $3$                                     \\ \hline
    $4$           &  $22$       & $10$ & $2$                                    \\ \hline
    $5$           &  $23$        & $11$ & $3$                                  \\ \hline                                              
    \end{tabular}
\end{center}
Since $V_n$ is calculated from the previous two terms, we see that $V_n \pmod{25}$ has period $10$.
When $p \equiv 1\pmod {5}$, we have $2p \equiv 2\pmod {10}$ and $2p-1 \equiv 1\pmod {10}$.
So by (\ref{eq:T}) it follows that
\[5T(p)\equiv 4V_2-2V_1+3 \equiv 4\cdot7-2\cdot 3+3\equiv 0 \pmod {25}.\]
Therefore $T(p)$ is divisible by $5$ when $p \equiv 1\pmod {5}$.

Next, we calculate $V_n \pmod{11}$ and obtain the following list.
\begin{center}
    \begin{tabular}{ | l | l | l | l| l|}
    \hline
    $n  $      & $V_n \pmod {11}$          \\ \hline
     $0$           &  $2$         \\ \hline
    $1$           &  $3$      \\ \hline
    $2$           &  $7$             \\ \hline
    $3$           &  $7$      \\ \hline
    $4$           &  $3$      \\ \hline
    $5$           &  $2$           \\ \hline   
    $6$           &  $3$           \\ \hline 
        \end{tabular}
\end{center}
Hence we see that $V_n \pmod {11}$ has  period $5$.
When $p\equiv 3 \pmod 5$, we have $2p\equiv 1 \pmod 5$ and $2p-1\equiv 0 \pmod 5$.
So it follows from (\ref{eq:T}) that
\[5T(p) \equiv 4V_1-2V_0+3 \equiv 4\cdot3-2\cdot 2+3\equiv 0 \pmod {11}.\]
Therefore $T(p)$ is divisible by $11$ when $p \equiv 3\pmod {5}$.

Finally, we calculate $V_n \pmod {31}$ and obtain the following list.
\begin{center}
    \begin{tabular}{ | l | l | l | l| l|}
    \hline
    $n  $      & $V_n \pmod {31}$  &  $n  $      & $V_n \pmod {31}$       \\ \hline
     $0$           &  $2$       & $9$    &  $12$                       \\ \hline
    $1$           &  $3$        & $10$   & $30$                                   \\ \hline
    $2$           &  $7$         & $11$ & $16$                                    \\ \hline
    $3$           &  $18$         & $12$ & $18$                                     \\ \hline
    $4$           &  $16$       & $13$ & $7$                                    \\ \hline
    $5$           &  $30$        & $14$ & $3$                                  \\ \hline   
    $6$           &  $12$       & $15$ & $2$                                    \\ \hline 
    $7$           &  $6$       & $16$ & $3$                                    \\ \hline    
    $8$           &  $6$       & &                                  \\ \hline                                    
    \end{tabular}
\end{center}
Hence we see that $V_n \pmod {31}$ has  period $15$.
When $p\equiv 2 \pmod {15}$, we have  $2p\equiv 4 \pmod {15}$ and  $2p-1\equiv 3 \pmod {15}$.
So it follows from (\ref{eq:T}) that
\[5T(p)\equiv 4V_4-2V_3+3 \equiv 4\cdot16-2\cdot 18+3\equiv 0 \pmod {31}.\]
Hence $T(p)$ is divisible by $31$ when $p \equiv 2\pmod {15}$
\end{proof}

By equation (\ref{eq:T}), it is easy to see that $T(p)$ goes to infinity as $p$ tends to infinity.
Also, by Dirichlet's theorem on arithmetic progressions, there are infinitely many primes $p$ of each of the  three forms in Theorem \ref{thm:T}.
Therefore, it follows that there are infinitely many prime numbers $p$ such that $T(p)$ is composite.

According to Theorem \ref{thm:T}, if $T(p)$ is prime for $p>5$, then $p$ is congruent to one of $7$, $19$, $29 \pmod{30}$.
However, there is no prime $q$ so that $q$ divides $T(p)$ for all $p$ congruent to any one of  $7$, $19$, $29$ $\pmod{30}$.
This can be easily seen by taking two primes in the same class and observing that the greatest common divisor of the two $T(p)$s is equal to $1$.

For other classes of $p$, there might be a trivial divisor as in Theorem \ref{thm:T} and can be proved by the same method.
For example, we can prove the following theorem.
\begin{theorem}
\begin{enumerate}
	\item If $p\equiv 7 \pmod{65}$, then $T(p)$ is divisible by $131$.
	\item If $p\equiv 29 \pmod{35}$, then $T(p)$ is divisible by $71$.
\end{enumerate}
\end{theorem}
\begin{proof}
The method of proof is exactly the same as that of in the proof of Theorem \ref{thm:T}.
Also since  periods are larger, we just sketch the proof.

One finds that $T(p) \pmod{131}$ has  period $65$.
When $p\equiv 7 \pmod{65}$, we have $2p\equiv 14 \pmod{65}$ and $2p-1\equiv 13 \pmod{65}$.
By direct calculations we see that $V_{14}\equiv 103 \pmod{131}$ and $V_{13}\equiv 11 \pmod{131}$.
Hence it follows from (\ref{eq:T}) that 
\[5T(p)\equiv4V_{14}-2V_{13}+3\equiv 4\cdot 103-2\cdot11+3\equiv 0 \pmod{131}.\]
Hence $T(p)$ is divisible by $131$.

The proof of the second statement is similar.
We see that $T(p) \pmod{71}$ has  period $35$.
When $p\equiv 29 \pmod{35}$, we have $2p\equiv 23 \pmod{35}$ and $2p-1\equiv 22 \pmod{35}$.
By direct calculations we see that $V_{23}\equiv 22 \pmod{71}$ and $V_{21}\equiv 10 \pmod{71}$.
Hence it follows from (\ref{eq:T}) that 
\[5T(p)\equiv4V_{23}-2V_{22}+3\equiv 4\cdot 22-2\cdot10+3\equiv 0 \pmod{71}.\]
Therefore $T(p)$ is divisible by $71$.
\end{proof}

Although we can find more similar divisibility properties, we refrain from stating them.
Finally, one can notice that we do not use the fact that $p$ is prime. 
All the arguments hold if one allows $p$ to be composite.

\providecommand{\bysame}{\leavevmode\hbox to3em{\hrulefill}\thinspace}
\providecommand{\MR}{\relax\ifhmode\unskip\space\fi MR }
\providecommand{\MRhref}[2]{%
  \href{http://www.ams.org/mathscinet-getitem?mr=#1}{#2}
}
\providecommand{\href}[2]{#2}

\end{document}